\documentclass[12pt, reqno]{amsart}
\usepackage{amsmath, amsthm, amscd, amsfonts, amssymb, graphicx, color,float,pgf,tikz}
\usepackage{amssymb,fontenc}
\usepackage{latexsym,wasysym,mathrsfs}
\usepackage{hyperref}
\usetikzlibrary{arrows}
\usepackage[square,numbers,sort&compress]{natbib}
\newtheorem{theorem}{Theorem}[section]
\newtheorem{lemma}[theorem]{Lemma}

\newtheorem{corollary}[theorem]{Corollary}

\theoremstyle{definition}
\newtheorem{definition}[theorem]{Definition}

\newtheorem{remark}[theorem]{Remark}
\numberwithin{equation}{section}

\newcommand{\be}{\begin{equation}}
\newcommand{\ee}{\end{equation}}

\numberwithin{equation}{section}
\usepackage{etoolbox}
\makeatletter
\patchcmd{\@settitle}{\uppercasenonmath\@title}{}{}{}
\patchcmd{\@setauthors}{\MakeUppercase}{}{}{}
\makeatother
\allowdisplaybreaks
\begin{document}
\setcounter{page}{1}
\title[ Perturbation and Stability of  Continuous Operator Frames in Hilbert $C^{\ast}$-Modules]{ Perturbation and Stability of  Continuous Operator Frames in Hilbert $C^{\ast}$-Modules}
\author[ Abdeslam Touri$^{*1}$, Hatim Labrigui$^{1}$ Mohamed Rossafi$^{2}$ and Samir Kabbaj$^{1}$]{Abdeslam Touri$^{*1}$, Hatim Labrigui$^{1}$ Mohamed Rossafi$^{2}$ and Samir Kabbaj$^{1}$}
\address{$^{1}$Department of Mathematics, Ibn Tofail University, B.P. 133, Kenitra, Morocco}
\email{\textcolor[rgb]{0.00,0.00,0.84}{  hlabrigui75@gmail;  touri.abdo68@gmail.com;samkabbaj@yahoo.fr}}
\address{$^{2}$LASMA Laboratory Department of Mathematics, Faculty of Sciences Dhar El Mahraz, University Sidi Mohamed Ben Abdellah, Fes, Morocco}
\email{\textcolor[rgb]{0.00,0.00,0.84}{rossafimohamed@gmail.com; mohamed.rossafi@usmba.ac.ma}}
\keywords{Continuous g-Frames, Controlled continuous g-frames, $C^{\ast}$-algebra, Hilbert $\mathcal{A}$-modules.\\
\indent
\\
\indent $^{*}$ Corresponding author}
\maketitle
\begin{abstract}
Frame Theory has a great revolution in recent years. This Theory have been extended from Hilbert spaces to Hilbert  $C^{\ast}$-modules. In this paper we consider the stability of continuous operator frame and continuous $K$-operator frames in Hilbert $C^{\ast}$-Modules under perturbation and we establish some properties.
\end{abstract}
\section{Introduction and preliminaries}
The concept of frames in Hilbert spaces is a new theory which was introduced by Duffin and Schaeffer \cite{Duf} in 1952 to study some deep problems in nonharmonic Fourier series. This theory was reintroduced and developed by Daubechies, Grossman and Meyer \cite{13}.

In 1993, S.T.Ali, J.P.Antoine and J.P.Gazeau \cite{STAJP} introduced the concept of continuous frames in Hilbert spaces. Gabardo and Han in \cite{14} called these kinds frames, frames associated with measurable spaces.

In 2000, Frank and Larson \cite{Frank} introduced the notion of frames in Hilbert $C^{\ast}$-modules as a generalization of frames in Hilbert spaces. The theory of continuous frames has been generalized in Hilbert $C^{\ast}$-modules. For more details, see \cite{mjpaa, mjpaaa, moi, moi1, ARAN, MR1, MR2, r1, r11, r2, r3, r4, r5, r6, r7, r10, r8, r9, LEC, R3}.\\
The aim of this paper is to extend results of M. Rossafi and A. Akhlidj \cite{r8}, given for Hilbert $C^{\ast}$-module in discret case.

%The goal of this article is the study of stability of continuous operator frame for $End^{\ast}_{\mathcal{A}}(\mathcal{H})$ and continuous $K$-operator frames for $End^{\ast}_{\mathcal{A}}(\mathcal{H})$ under perturbation and we establish some properties. 

In the following we briefly recall the definitions and basic properties of $C^{\ast}$-algebra and Hilbert $\mathcal{A}$-modules. Our references for $C^{\ast}$-algebras are \cite{{Dav},{Con}}. For a $C^{\ast}$-algebra $\mathcal{A}$, if $a\in\mathcal{A}$ is positive we write $a\geq 0$ and $\mathcal{A}^{+}$ denotes the set of positive elements of $\mathcal{A}$.
\begin{definition}\cite{Con}.	
	Let $ \mathcal{A} $ be a unital $C^{\ast}$-algebra and $\mathcal{H}$ be a left $ \mathcal{A} $-module, such that the linear structures of $\mathcal{A}$ and $ \mathcal{H} $ are compatible. $\mathcal{H}$ is a pre-Hilbert $\mathcal{A}$-module if $\mathcal{H}$ is equipped with an $\mathcal{A}$-valued inner product $\langle.,.\rangle_{\mathcal{A}} :\mathcal{H}\times\mathcal{H}\rightarrow\mathcal{A}$, such that is sesquilinear, positive definite and respects the module action. In the other words,
	\begin{itemize}
		\item [(i)] $ \langle x,x\rangle_{\mathcal{A}}\geq0 $, for all $ x\in\mathcal{H} $, and $ \langle x,x\rangle_{\mathcal{A}}=0$ if and only if $x=0$.
		\item [(ii)] $\langle ax+y,z\rangle_{\mathcal{A}}=a\langle x,z\rangle_{\mathcal{A}}+\langle y,z\rangle_{\mathcal{A}},$ for all $a\in\mathcal{A}$ and $x,y,z\in\mathcal{H}$.
		\item[(iii)] $ \langle x,y\rangle_{\mathcal{A}}=\langle y,x\rangle_{\mathcal{A}}^{\ast} $, for all $x,y\in\mathcal{H}$.
	\end{itemize}	 
	For $x\in\mathcal{H}, $ we define $||x||=||\langle x,x\rangle_{\mathcal{A}}||^{\frac{1}{2}}$. If $\mathcal{H}$ is complete with $||.||$, it is called a Hilbert $\mathcal{A}$-module or a Hilbert $C^{\ast}$-module over $\mathcal{A}$.\\
	 For every $a$ in $C^{\ast}$-algebra $\mathcal{A}$, we have $|a|=(a^{\ast}a)^{\frac{1}{2}}$ and the $\mathcal{A}$-valued norm on $\mathcal{H}$ is defined by $|x|=\langle x, x\rangle_{\mathcal{A}}^{\frac{1}{2}}$, for all $x\in\mathcal{H}$.
	
	Let $\mathcal{H}$ and $\mathcal{K}$ be two Hilbert $\mathcal{A}$-modules, a map $T:\mathcal{H}\rightarrow\mathcal{K}$ is said to be adjointable if there exists a map $T^{\ast}:\mathcal{K}\rightarrow\mathcal{H}$ such that $\langle Tx,y\rangle_{\mathcal{A}}=\langle x,T^{\ast}y\rangle_{\mathcal{A}}$ for all $x\in\mathcal{H}$ and $y\in\mathcal{K}$.
	
We reserve the notation $End_{\mathcal{A}}^{\ast}(\mathcal{H},\mathcal{K})$ for the set of all adjointable operators from $\mathcal{H}$ to $\mathcal{K}$ and $End_{\mathcal{A}}^{\ast}(\mathcal{H},\mathcal{H})$ is abbreviated to $End_{\mathcal{A}}^{\ast}(\mathcal{H})$.

\end{definition}

The following lemmas will be used to prove our mains results.
\begin{lemma} \label{1} \cite{Pas}.
	Let $\mathcal{H}$ be a Hilbert $\mathcal{A}$-module. If $T\in End_{\mathcal{A}}^{\ast}(\mathcal{H})$, then $$\langle Tx,Tx\rangle_{\mathcal{A}}\leq\|T\|^{2}\langle x,x\rangle_{\mathcal{A}}, \qquad x\in\mathcal{H}.$$
\end{lemma}

\begin{lemma} \label{sb} \cite{Ara}.
	Let $\mathcal{H}$ and $\mathcal{K}$ be two Hilbert $\mathcal{A}$-modules and $T\in End_{\mathcal{A}}^{\ast}(\mathcal{H},\mathcal{K})$. Then the following statements are equivalent:
	\begin{itemize}
		\item [(i)] $T$ is surjective.
		\item [(ii)] $T^{\ast}$ is bounded below with respect to norm, i.e., there is $m>0$ such that $\|T^{\ast}x\|\geq m\|x\|$, for all $x\in\mathcal{K}$.
		\item [(iii)] $T^{\ast}$ is bounded below with respect to the inner product, i.e., there is $m'>0$ such that $\langle T^{\ast}x,T^{\ast}x\rangle_{\mathcal{A}}\geq m'\langle x,x\rangle_{\mathcal{A}}$, for all $x\in\mathcal{K}$.
	\end{itemize}
\end{lemma}
\begin{lemma} \label{33} \cite{33}.
	Let $(\Omega,\mu )$ be a measure space, $X$ and $Y$ are two Banach spaces, $\lambda : X\longrightarrow Y$ be a bounded linear operator and $f : \Omega\longrightarrow X$ measurable function; then, 
	\begin{equation*}
	\lambda (\int_{\Omega}fd\mu)=\int_{\Omega}(\lambda f)d\mu.
	\end{equation*}
\end{lemma}

\section{Characterisation of continuous operator frame for $ End_{\mathcal{A}}^{\ast}(\mathcal{H})$ }
Let $X$ be a Banach space, $(\Omega,\mu)$ a measure space, and $f:\Omega\to X$ be a measurable function. Integral of Banach-valued function $f$ has been defined by Bochner and others. Most properties of this integral are similar to those of the integral of real-valued functions (see \cite{32, 33}). Since every $C^{\ast}$-algebra and Hilbert $C^{\ast}$-module are Banach spaces, we can use this integral and its properties.

Let $(\Omega,\mu)$ be a measure space, $U$ and $V$ be two Hilbert $C^{\ast}$-modules over a unital $C^{\ast}$-algebra and $\{V_{w}\}_{w\in\Omega}$  is a family of submodules of $V$.  $End_{\mathcal{A}}^{\ast}(U,V_{w})$ is the collection of all adjointable $\mathcal{A}$-linear maps from $U$ into $V_{w}$.

We define, following:
\begin{equation*}
	l^{2}(\Omega, \{V_{w}\}_{\omega \in \Omega})=\left\{x=\{x_{w}\}_{w\in\Omega}: x_{w}\in V_{w}, \left\|\int_{\Omega}|x_{w}|^{2}d\mu(w)\right\|<\infty\right\}.
\end{equation*}
For any $x=\{x_{w}\}_{w\in\Omega}$ and $y=\{y_{w}\}_{w\in\Omega}$, the $\mathcal{A}$-valued inner product is defined by $\langle x,y\rangle=\int_{\Omega}\langle x_{w},y_{w}\rangle_{\mathcal{A}} d\mu(w)$ and the norm is defined by $\|x\|=\|\langle x,x\rangle\|^{\frac{1}{2}}$. In this case the $l^{2}(\Omega,\{V_{w}\}_{\omega \in \Omega})$ is an Hilbert $C^{\ast}$-module (see \cite{28}).

\begin{definition}
	We call $\Lambda :=\{\Lambda_{w}\in End_{\mathcal{A}}^{\ast}(\mathcal{H}): w\in\Omega\}$ a continuous operator frame for $End_{\mathcal{A}}^{\ast}(\mathcal{H})$ if:\\
(a) for any $x\in \mathcal{H}$, the mapping $\tilde{x}:\Omega\rightarrow V_{w}$ defined by $\tilde{x}(w)=\Lambda_{w}x$ is measurable;
(b) there is a pair of constants $0<A, B$ such that, for any $x\in\mathcal{H}$,
		\begin{equation} \label{2.1}
			A\langle x,x\rangle_{\mathcal{A}} \leq\int_{\Omega}\langle\Lambda_{w}x,\Lambda_{w}x\rangle_{\mathcal{A}} d\mu(w)\leq B\langle x,x\rangle_{\mathcal{A}}, \quad  x\in \mathcal{H}.
		\end{equation}
The constants $A$ and $B$ are called continuous operator frame bounds.

If $A=B$ we call this continuous operator frame a continuous  tight operator frame, and if $A=B=1$ it is called a continuous Parseval operator frame.

If only the right-hand inequality of \eqref{2.1} is satisfied, we call $\Lambda =\{\Lambda_{w}\}_{w\in\Omega}$ the continuous Bessel operator frame for $End_{\mathcal{A}}^{\ast}(\mathcal{H})$  with Bessel bound $B$.\\

	The continuous frame operator $S$ of $\Lambda$ on $\mathcal{H}$ is defined by :
	\begin{equation*}
	Sx=\int_{\Omega}\Lambda^{\ast}_{\omega}\Lambda_{\omega}xd\mu (\omega), \quad  x\in \mathcal{H}.
	\end{equation*}
	The continuous frame operator S is a bounded, positive, selfadjoint, and invertible.
\end{definition}

\begin{theorem}
Let $\Lambda= \{\Lambda_{w}\in End_{\mathcal{A}}^{\ast}(\mathcal{H}): w\in\Omega\}$. $\Lambda$ is a continuous operator frame for $End_{\mathcal{A}}^{\ast}(\mathcal{H})$  if and only if there exist a constants $0<A, B$ such that for any $x\in \mathcal{H}$ :
\begin{equation*} \label{to1}
	A\|x\|^{2} \leq\| \int_{\Omega}\langle\Lambda_{w}x,\Lambda_{w}x\rangle_{\mathcal{A}} d\mu(w)\| \leq B\|x\|^{2}.
\end{equation*}
\end{theorem}
\section{Perturbation and stability  of continuous operator frame for $ End_{\mathcal{A}}^{\ast}(\mathcal{H})$ }
\begin{theorem}
	Let $ \{T_w\}_{w \in \Omega}$ be a continuous operator frame for $End_{\mathcal{A}}^{\ast}(\mathcal{H})$ with bounds A and B. If $\{R_w\}_{w \in \Omega} \subset End_{\mathcal{A}}^{\ast}(\mathcal{H}) $ is a continuous operator Bessel family with bound $M<A$, then $ \{T_w \mp R_w\}_{w \in \Omega}$ is a continuous operator frame for $End_{\mathcal{A}}^{\ast}(\mathcal{H})$. 

\end{theorem}
\begin{proof}
	We just proof the case that $ \{T_w + R_w\}_{w \in \Omega}$ is a continuous operator frame for $End_{\mathcal{A}}^{\ast}(\mathcal{H})$. \\
	On one hand, For each $x \in \mathcal{H}$, we have
\begin{align*}
   \|\{(T_w + R_w)x\}_{w \in \Omega}\| &=\|\int_{\Omega}\langle (T_w + R_w)x,(T_w + R_w)x\rangle_{\mathcal{A}} d\mu(w)\|^{\frac{1}{2}}\\
    &\leq \|\{T_wx\}_{w \in \Omega}\|+\|\{ R_wx\}_{w \in \Omega}\|\\
    &= \|\int_{\Omega}\langle T_w x,T_w x\rangle_{\mathcal{A}} d\mu(w)\|^{\frac{1}{2}} + \|\int_{\Omega}\langle R_w x, R_w x\rangle_{\mathcal{A}} d\mu(w)\|^{\frac{1}{2}}\\
    &\leq \sqrt{B}\|x\|+\sqrt{M}\|x\|.
\end{align*}
    Hence
 \begin{equation}\label{haja1}
     \|\int_{\Omega}\langle (T_w + R_w)x,(T_w + R_w)x\rangle_{\mathcal{A}} d\mu(w)\|^{\frac{1}{2}}\leq (\sqrt{B}+\sqrt{M})\|x\|.
 \end{equation}
    One the other hand we have
    \begin{align*}
   \|\{(T_w + R_w)x\}_{w \in \Omega}\|&=\|\int_{\Omega}\langle (T_w + R_w)x,(T_w + R_w)x\rangle_{\mathcal{A}} d\mu(w)\|^{\frac{1}{2}}\\
    &\geq \|\{T_wx\}_{w \in \Omega}\|-\|\{ R_wx\}_{w \in \Omega}\|\\
    &=\|\int_{\Omega}\langle T_w x,T_w x\rangle_{\mathcal{A}} d\mu(w)\|^{\frac{1}{2}}-\|\int_{\Omega}\langle  R_{w}x, R_{w}x\rangle_{\mathcal{A}} d\mu(w)\|^{\frac{1}{2}}\\
    &\geq \sqrt{A}\|x\|-\sqrt{M}\|x\|.
    \end{align*}
     Then
    \begin{equation}\label{haja2}
    \|\int_{\Omega}\langle (T_w + R_w)x,(T_w + R_w)x\rangle_{\mathcal{A}} d\mu(w)\|^{\frac{1}{2}}\geq  (\sqrt{A}-\sqrt{M})\|x\|.
    \end{equation}
    From (\ref{haja1}) and (\ref{haja2}) we get
    $$(\sqrt{A}-\sqrt{M})^2\|x\|^2\leq\|\int_{\Omega}\langle (T_w + R_w)x,(T_w + R_w)x\rangle_{\mathcal{A}} d\mu(w)\|\leq (\sqrt{B}+\sqrt{M})^2\|x\|^2 .$$
	Therefore $ \{T_w + R_w\}_{w \in \Omega}$ is a continuous operator frame for $End_{\mathcal{A}}^{\ast}(\mathcal{H})$.
	
\end{proof}
\begin{theorem}
	Let $ \{T_w \}_{w \in \Omega}$ be a continuous operator frame for $End_{\mathcal{A}}^{\ast}(\mathcal{H})$ with bounds $A$ and $B$ and let $ \{R_w\}_{w \in \Omega} \subset End_{\mathcal{A}}^{\ast}(\mathcal{H})$. The following statements are equivalent:
\begin{itemize}
	\item [(i)] $ \{R_w\}_{w \in \Omega}$ is a continuous operator frame for $End_{\mathcal{A}}^{\ast}(\mathcal{H})$ .
	\item [(ii)] There exists a constant $M>0$, such that for all x in $\mathcal{H} $
    we have 
	 \begin{equation}\label{haja3}
	\scriptsize 
	\|\int_{\Omega}\langle (T_w - R_w)x,(T_w - R_w)x\rangle_{\mathcal{A}} d\mu(w)\|\leq M. min(\|\int_{\Omega}\langle T_w x,T_w x\rangle_{\mathcal{A}} d\mu(w)\|, \|\int_{\Omega}\langle R_w x,R_w x\rangle_{\mathcal{A}} d\mu(w)\|).
	\end{equation}
\end{itemize}
\end{theorem}
\begin{proof}
	Suppose that $ \{R_w\}_{w \in \Omega}$  is a continuous operator frame for $End_{\mathcal{A}}^{\ast}(\mathcal{H})$ with bound C and D. Then for all $x \in \mathcal{H} $ we have 
\begin{align*}
	\|\{(T_w - R_w)x\}_{w \in \Omega}\|&=\|\int_{\Omega}\langle (T_w - R_w)x,(T_w - R_w)x\rangle_{\mathcal{A}} d\mu(w)\|^{\frac{1}{2}}\\
	&\leq \|\{T_wx\}_{w \in \Omega}\|+\|\{ R_wx\}_{w \in \Omega}\|\\
	&=\|\int_{\Omega}\langle T_w x,T_w x\rangle_{\mathcal{A}} d\mu(w)\|^{\frac{1}{2}}+\|\int_{\Omega}\langle R_wx, R_wx\rangle_{\mathcal{A}} d\mu(w)\|^{\frac{1}{2}}\\
	&\leq \|\int_{\Omega}\langle T_w x,T_w x\rangle_{\mathcal{A}} d\mu(w)\|^{\frac{1}{2}}+ \sqrt{D}\|x\|\\
	&\leq  \|\int_{\Omega}\langle T_w x,T_w x\rangle_{\mathcal{A}} d\mu(w)\|^{\frac{1}{2}}+\sqrt{\frac{D}{A}}  \|\int_{\Omega}\langle T_w x,T_w x\rangle_{\mathcal{A}} d\mu(w)\|^{\frac{1}{2}}\\
	&= (1+\sqrt{\frac{D}{A}}) \|\int_{\Omega}\langle T_w x,T_w x\rangle_{\mathcal{A}} d\mu(w)\|^{\frac{1}{2}}.
\end{align*}
    In the same way we have 
    $$\|\int_{\Omega}\langle (T_w - R_w)x,(T_w - R_w)x\rangle_{\mathcal{A}} d\mu(w)\|^{\frac{1}{2}}\leq \left(1+\sqrt{\frac{B}{C}}\right) \|\int_{\Omega}\langle R_w x,R_w x\rangle_{\mathcal{A}} d\mu(w)\|^{\frac{1}{2}}.$$
    For (\ref{haja3}), we take $M= min (1+\sqrt{\frac{B}{C}} ,1+\sqrt{\frac{D}{A}} )$.\\
    Now we assume that (\ref{haja3}) holds. For each $x \in \mathcal{H} $, we have: 
\begin{align*}
    \sqrt{A}\|x\| &\leq\|\int_{\Omega}\langle T_w x,T_w x\rangle_{\mathcal{A}} d\mu(w)\|^{\frac{1}{2}}\\
    &=\|\{T_wx\}_{w \in \Omega}\|\\
    &\leq \|\{(T_w - R_w)x\}_{w \in \Omega}\|+\|\{ R_wx\}_{w \in \Omega}\|\\
    &=  \|\int_{\Omega}\langle (T_w - R_w)x,(T_w - R_w)x\rangle_{\mathcal{A}} d\mu(w)\|^{\frac{1}{2}}+\|\int_{\Omega}\langle R_w x,R_w x\rangle_{\mathcal{A}} d\mu(w)\|^{\frac{1}{2}}
\end{align*}
    From ( \ref{haja3}), we have 
    
    $$	\|\int_{\Omega}\langle (T_w - R_w)x,(T_w - R_w)x\rangle_{\mathcal{A}} d\mu(w)\|\leq M \|\int_{\Omega}\langle R_w x,R_w x\rangle_{\mathcal{A}} d\mu(w)\|.$$
    Then
    $$\|\int_{\Omega}\langle T_w x,T_w x\rangle_{\mathcal{A}} d\mu(w)\|^{\frac{1}{2}} \leq \sqrt{M}\|\int_{\Omega}\langle R_w x,R_w x\rangle_{\mathcal{A}} d\mu(w)\|^{\frac{1}{2}}+\| \int_{\Omega}\langle R_w x,R_w x\rangle_{\mathcal{A}} d\mu(w)\|^{\frac{1}{2}}. $$
    
    Hence 
 \begin{equation} \label{haja4}
    \sqrt{A}\|x\| \leq\|\int_{\Omega}\langle T_w x,T_w x\rangle_{\mathcal{A}} d\mu(w)\|^{\frac{1}{2}}\leq (1+\sqrt{M})\| \int_{\Omega}\langle R_w x,R_w x\rangle_{\mathcal{A}} d\mu(w)\|^{\frac{1}{2}}.
 \end{equation}
    Also, we have 
\begin{align*}
    \|\{ R_wx\}_{w \in \Omega}\|&=\|\int_{\Omega}\langle R_w x,R_w x\rangle_{\mathcal{A}} d\mu(w)\|^{\frac{1}{2}}\\
    &= \|\{ (R_wx - T_wx) + T_wx\}_{w \in \Omega}\|\\
    &\leq \|\{(T_w - R_w)x\}_{w \in \Omega}\|+\|\{ T_wx\}_{w \in \Omega}\|\\  
    &= \|\int_{\Omega}\langle (T_w - R_w)x,(T_w - R_w)x\rangle_{\mathcal{A}} d\mu(w)\|^{\frac{1}{2}}+ \|\int_{\Omega}\langle T_w x,T_w x\rangle_{\mathcal{A}} d\mu(w)\|^{\frac{1}{2}}\\    
\end{align*}
    From (\ref{haja3}), we have 
    
    $$	\|\int_{\Omega}\langle (T_w - R_w)x,(T_w - R_w)x\rangle_{\mathcal{A}} d\mu(w)\|\leq M \|\int_{\Omega}\langle T_w x,T_w x\rangle_{\mathcal{A}} d\mu(w)\|.$$
    Then 

    $$\|\int_{\Omega}\langle R_w x,R_w x\rangle_{\mathcal{A}} d\mu(w)\|^{\frac{1}{2}} \leq (1+\sqrt{M})\|\int_{\Omega}\langle T_w x,T_w x\rangle_{\mathcal{A}} d\mu(w)\|^{\frac{1}{2}}.$$
    So, 
  \begin{equation}\label{haja5}
    \|\int_{\Omega}\langle R_w x,R_w x\rangle_{\mathcal{A}} d\mu(w)\|^{\frac{1}{2}} \leq (1+\sqrt{M})\sqrt{B} \|x\|.
 \end{equation}
    From (\ref{haja4})  and (\ref{haja5}) we give that 
    $$\frac{A}{(1+\sqrt{M})^2} \|x\|^2 \leq \|\int_{\Omega}\langle R_w x,R_w x\rangle_{\mathcal{A}} d\mu(w)\|\leq B (1+\sqrt{M})^2\|x\|^2.$$
    Therefore $ \{R_w\}_{w \in \Omega}$ is a continuous operator frame for $End_{\mathcal{A}}^{\ast}(\mathcal{H})$. 
\end{proof}
\begin{theorem}
	Let $ \{T_{k,w}\}_{w \in \Omega} \subset End_{\mathcal{A}}^{\ast}(\mathcal{H})$, $k=1,2,...,n$ be a continuous operator frames for $End_{\mathcal{A}}^{\ast}(\mathcal{H})$ with bounds $A_k$ and $B_k$ and let $\{\alpha_k\}_{k\in \{1,...,n\}}$ be any scalars. If there exists a constant $\lambda >0$ and some $p\in\{1,2,...,n\}$ such that
	$$\lambda\|\{T_{p,w}\}_{w \in \Omega} \|\leq \|\sum_{k=1}^{n}\alpha_kT_{k,w}x \|,\qquad  x  \in \mathcal{H}. $$
	Then $\{\sum_{k=1}^{n}\alpha_kT_{k,w}\}_{w \in \Omega}$ is a continuous operator frame for $End_{\mathcal{A}}^{\ast}(\mathcal{H})$. And conversely.
	
\end{theorem}
\begin{proof}
	For every $x  \in \mathcal{H}$, we have 
\begin{align*}
    \sqrt{A_p} \lambda \|\langle  x, x\rangle_{\mathcal{A}}\|^{\frac{1}{2}}
    &\leq \| \{T_{p,w}x\}_{w \in \Omega}\|\\
    &\leq \|\{\sum_{k=1}^{n}\alpha_kT_{k,w}x\} _{w \in \Omega}\|\\
     &\leq \sum_{k=1}^{n}|\alpha_k| \|\{T_{k,w}x\} _{w \in \Omega}\|\\
    &\leq (\max_{1 \leq k\leq n } |\alpha_k|) \sum_{k=1}^{n}\|\{T_{k,w}x\} _{w \in \Omega}\|\\
    &\leq (\max_{1 \leq k\leq n }  |\alpha_k|) (\sum_{k=1}^{n}\sqrt{B_k}) \|\langle  x, x\rangle_{\mathcal{A}}\|^{\frac{1}{2}}.
\end{align*}
   Hence 
   $${A_p} {\lambda }^2\|\langle  x, x\rangle_{\mathcal{A}}\|\leq \|\{\sum_{k=1}^{n}\alpha_kT_{k,w}x\} _{w \in \Omega}\|^2 \leq (\max_{1 \leq k\leq n }  |\alpha_k|)^2 (\sum_{k=1}^{n}\sqrt{B_k})^2 \|\langle  x, x\rangle_{\mathcal{A}}\| .$$
   Therefore $\{\sum_{k=1}^{n}\alpha_kT_{k,w}x\} _{w \in \Omega}$ is a continuous operator frame for $End_{\mathcal{A}}^{\ast}(\mathcal{H})$.\\
   For the converse, let $\{\sum_{k=1}^{n}\alpha_kT_{k,w}x\} _{w \in \Omega}$ be a continuous operator frame for $End_{\mathcal{A}}^{\ast}(\mathcal{H})$ with bounds A , B and let any $k \in \{1,2,...,n\}$.\\
   Since  $\{T_{p,w}\} _{w \in \Omega}$ is a continuous operator frame for $End_{\mathcal{A}}^{\ast}(\mathcal{H})$ with bounds
   $A_p$ and  $B_p$, then for any $x \in \mathcal{H}$ we have, 
   $$A_p\|\langle  x, x\rangle_{\mathcal{A}}\|\leq \|\{T_{p,w}\} _{w \in \Omega}\|^2\leq B_p \|\langle  x, x\rangle_{\mathcal{A}}\|.$$
   Hence,
   $$B_p^{-1}\|\{T_{p,w}\} _{w \in \Omega}\|^2\leq \|\langle  x, x\rangle_{\mathcal{A}}\|.$$
   Also, we have
   $$A\|\langle  x, x\rangle_{\mathcal{A}}\|\leq \|\{\sum_{k=1}^{n}\alpha_kT_{k,w}x\} _{w \in \Omega}\|^2, \qquad x \in \mathcal{H}. $$
   Then, 
   $$\|\langle  x, x\rangle_{\mathcal{A}}\|\leq A^{-1}\|\{\sum_{k=1}^{n}\alpha_kT_{k,w}x\} _{w \in \Omega}\|^2, \qquad x \in \mathcal{H}.$$
   So, 
   $$\frac{A}{B_p} \|\{T_{p,w}\} _{w \in \Omega}\|^2 \leq \|\{\sum_{k=1}^{n}\alpha_kT_{k,w}x\} _{w \in \Omega}\|^2, \qquad x \in \mathcal{H}.$$
   Then for $\lambda=\frac{A}{B_p}$ we have 
   $$\lambda \|\{T_{p,w}\} _{w \in \Omega}\|^2\leq \|\{\sum_{k=1}^{n}\alpha_kT_{k,w}x\} _{w \in \Omega}\|^2, \qquad x \in \mathcal{H}.
    $$
    which ends the proof
\end{proof}
\begin{theorem}
	For $k=1,2,...,n$, let $\{T_{k,w}\} _{w \in \Omega} \subset End_{\mathcal{A}}^{\ast}(\mathcal{H})$ be a continuous operator frames with bounds  $A_k$ and  $B_k$ and let $\{R_{k,w}\} _{w \in \Omega}\subset End_{\mathcal{A}}^{\ast}(\mathcal{H})$.\\ Let $L: l^2(\mathcal{H})\longrightarrow l^2(\mathcal{H}) $ be a bounded linear operator such that:
	\begin{equation*}
	L(\{\sum_{k=1}^{n}R_{k,w}\} _{w \in \Omega})=\{T_{p,w}\} _{w \in \Omega} \quad for some \quad p \in \{1,2,..,n\}.
	\end{equation*}
	 If there exists a constant $\lambda >0$ such that for each $x\in \mathcal{H}$ and $ k=1,..,n$ we have,

	\begin{equation*}
	\| \int_{\Omega}\langle (T_{k,w}-R_{k,w}) x,(T_{k,w}-R_{k,w}) x\rangle_{\mathcal{A}} d\mu(w)\| \leq \lambda \|\int_{\Omega}\langle T_{k,w} x,T_{k,w} x\rangle_{\mathcal{A}} d\mu(w)\|.
	\end{equation*}
	Then $\{\sum_{k=1}^{n}R_{k,w}\} _{w \in \Omega}$ is a continuous operator frame for $End_{\mathcal{A}}^{\ast}(\mathcal{H})$.
\end{theorem}
\begin{proof}
	For all $x \in \mathcal{H}$, we have
\begin{align*}
    \|\{\sum_{k=1}^{n}R_{k,w}x\} _{w \in \Omega}\|&\leq \sum_{k=1}^{n}\|\{R_{k,w}x\} _{w \in \Omega}\|\\
    &\leq \sum_{k=1}^{n}(\|\{T_{k,w}-R_{k,w}x\}_{w \in \Omega}\|+\|\{T_{k,w}x\}_{w \in \Omega}\|)\\
    &\leq (1+\sqrt{\lambda}) \|\sum_{k=1}^{n} \|\{T_{k,w}x\} _{w \in \Omega}\|\\
    &\leq  (1+\sqrt{\lambda})(\sum_{k=1}^{n}\sqrt{B_k}) \|\langle  x, x\rangle_{\mathcal{A}}\|^{\frac{1}{2}}.
\end{align*}
    Since, for any $x\in \mathcal{H}$, we have 
    $$\|L(\{\sum_{k=1}^{n}R_{k,w}\} _{w \in \Omega})\|=\|\{T_{p,w}\} _{w \in \Omega}\|$$
    Then
\begin{align*}
   \sqrt{A_p}\|\langle  x, x\rangle_{\mathcal{A}}\|^{\frac{1}{2}}&\leq \|\{T_{p,w}\} _{w \in \Omega}\|\\
   &= \|L(\{\sum_{k=1}^{n}R_{k,w}\} _{w \in \Omega})\|\\
   &\leq \|L\| \|\{\sum_{k=1}^{n}R_{k,w}\} _{w \in \Omega}\|.
\end{align*}
	Hence
	$$\frac{\sqrt{A_p}}{\|L\|}\|\langle  x, x\rangle_{\mathcal{A}}\|^{\frac{1}{2}}\leq \|\{\sum_{k=1}^{n}R_{k,w}\} _{w \in \Omega}\|,\qquad  x\in \mathcal{H} $$
	Therefore 
	$$\frac{\sqrt{A_p}}{\|L\|}\|\langle  x, x\rangle_{\mathcal{A}}\|^{\frac{1}{2}}\leq \|\{\sum_{k=1}^{n}R_{k,w}\} _{w \in \Omega}\|\leq (1+\sqrt{\lambda})(\sum_{k=1}^{n}\sqrt{B_k}) \|\langle  x, x\rangle_{\mathcal{A}}\|^{\frac{1}{2}}.$$
	This give that $\{\sum_{k=1}^{n} R_{k,w}\} _{w \in \Omega}$ is a continuous operator frame for $End_{\mathcal{A}}^{\ast}(\mathcal{H})$.
	
\end{proof}
\section{Characterisation of continuous K-operator frames for $ End_{\mathcal{A}}^{\ast}(\mathcal{H})$ }
  
\begin{definition}
	Let $K \in End_{\mathcal{A}}^{\ast}(\mathcal{H})$. A family of adjointable operators $\{T_{w}\}_{w \in \Omega}$ on a Hilbert $\mathcal{A}$-module $\mathcal{H}$ is said a continuous K-operator frame for $End_{\mathcal{A}}^{\ast}(\mathcal{H})$, if there exists two positive constants $A,B>0$ such that 
\begin{equation}\label{haja13}
    A\langle K^{\ast} x, K^{\ast} x\rangle_{\mathcal{A}}\leq \int_{\Omega}\langle T_{w} x,T_{w} x\rangle_{\mathcal{A}} d\mu({\omega})\leq B\langle  x, x\rangle_{\mathcal{A}}, \qquad x\in \mathcal{H}.
\end{equation}
    The numbers A and B are called respectively lower and upper bound of the continuous K-operator frame.\\
     The continuous K-operator frame is called a A-thight if:
     \begin{equation*}
     A\langle K^{\ast} x, K^{\ast} x\rangle_{\mathcal{A}}= \int_{\Omega}\langle T_{w} x,T_{w} x\rangle_{\mathcal{A}} d\mu({\omega}) 
     \end{equation*}
     If $A=1$, it is called a normalised tight continuous K-operator frame or a parseval continuous K-operator frame. The continuous K-operator frame is standard if for every $x\in \mathcal{H}$, the sum (\ref{haja13}) converges in norm.\\
\end{definition}
    \begin{remark}
	For any $K\in End_{\mathcal{A}}^{\ast}(\mathcal{H})$, every continuous operator frame is a continuous K-operator frame.\\
	Indeed, for any $K\in End_{\mathcal{A}}^{\ast}(\mathcal{H})$ we have, 
	$$\langle K^{\ast} x, K^{\ast} x\rangle_{\mathcal{A}}\leq \|K\|^2\langle x,x\rangle_{\mathcal{A}}, \qquad x\in \mathcal{H}.$$

	Let  $\{T_{w}\}_{w \in \Omega}$ be a continuous operator frame with bounds A and B then
	$$A \|K\|^{-2} \langle K^{\ast} x, K^{\ast} x\rangle_{\mathcal{A}} \leq A \langle  x, x\rangle_{\mathcal{A}}\leq \int_{\Omega}\langle T_{w} x,T_{w} x\rangle_{\mathcal{A}} d\mu({\omega})\leq B\langle  x, x\rangle_{\mathcal{A}},x\in \mathcal{H}. $$
	Hence $\{T_{w}\}_{w \in \Omega}$ is a continuous K-operator frame with bounds $A \|K\|^{-2}$ and B.\\
	\end{remark}
    Let  $\{T_{w}\}_{w \in \Omega}$ be a continuous K-operator for $End_{\mathcal{A}}^{\ast}(\mathcal{H})$.
    We define the operator
    \begin{align*}
    \mathcal{R}:\mathcal{H} &\longrightarrow l^2(\mathcal{H})\\
    x&\longrightarrow \mathcal{R}x=\{T_{w}x\}_{w \in \Omega}.
    \end{align*}
  
    The operator $\mathcal{R}$ is called the analysis operator of the continuous K-operator frame $\{T_{w}\}_{w \in \Omega}$, its adjoint is defined as follow,
   \begin{align*}
 \mathcal{R}^{\ast}: l^2(\mathcal{H})&\longrightarrow \mathcal{H}\\
 \{x_{w}\}_{w \in \Omega}&\longrightarrow \mathcal{R}^{\ast}(\{x_{w}\}_{w \in \Omega}) =\int_{\Omega} T_{w}^{\ast} x_{w} d\mu({\omega}).
   \end{align*}

    The operators $\mathcal{R}$ is called the synthesis operator of the continuous K-operator frame $\{T_{w}\}_{w \in \Omega}$.\\ By composing $\mathcal{R}$ and  $\mathcal{R}^{\ast}$ we obtain the operator 
    \begin{align*}
    S_{\mathcal{K}}:\mathcal{H} &\longrightarrow \mathcal{H}\\
    x&\longrightarrow S_{\mathcal{K}}(x)=\mathcal{R}^{\ast}\mathcal{R}x=\int_{\Omega} T_{w}^{\ast}T_{w} x d\mu({\omega})
    \end{align*}
    It's easy to show that the operator $ S_{\mathcal{K}}$ is positive and selfadjoint.
 \begin{theorem}
	Let $\{T_{w}\}_{w \in \Omega}$ be a family of adjointable operators on a Hilbert $\mathcal{A}$-module $\mathcal{H}$. Assume that $\int_{\Omega}\langle T_{w} x,T_{w} x\rangle_{\mathcal{A}} d\mu({\omega})$ converge in norm for all $x\in \mathcal{H}$. Then $\{T_{w}\}_{w \in \Omega}$ is a continuous K-operator frame for $End_{\mathcal{A}}^{\ast}(\mathcal{H})$ if and only if there exists two positive constants $A,B>0$ such that 
\begin{equation} \label{haja17}
    A\|K^{\ast} x\|^2 \leq \|\int_{\Omega}\langle T_{w} x,T_{w} x\rangle_{\mathcal{A}} d\mu({\omega})\|\leq B \| x\|^2.
\end{equation}
\end{theorem}
\begin{proof}
    Suppose that $\{T_{w}\}_{w \in \Omega}$ is a continuous K-operator frame .\\ From the definition of continuous K-operator frame (\ref{haja17}) holds.\\
    Conversely, assume that (\ref{haja17}) holds. The frame operator $S_{\mathcal{K}}$ is positive and selfadjoint, then 
    $$\langle S_{\mathcal{K}}^{\frac{1}{2}} x, S_{\mathcal{K}}^{\frac{1}{2}} x\rangle_{\mathcal{A}}=\langle S_{\mathcal{K}} x,  x\rangle_{\mathcal{A}}=\int_{\Omega}\langle T_{w} x,T_{w} x\rangle_{\mathcal{A}} d\mu({\omega}).$$
    We have for any $x\in \mathcal{H}$
    $$\sqrt{A}\|K^{\ast} x\|\leq \| S_{\mathcal{K}}^{\frac{1}{2}} x\|\leq \sqrt{B}\| x\|.$$
     Using lemma (\ref{sb}) , there exist two constants $m,M>0$ such that 
    $$m\langle K^{\ast} x, K^{\ast} x\rangle_{\mathcal{A}} \leq \langle S_{\mathcal{K}}^{\frac{1}{2}} x, S_{\mathcal{K}}^{\frac{1}{2}} x\rangle_{\mathcal{A}} = \int_{\Omega}\langle T_{w} x,T_{w} x\rangle_{\mathcal{A}} d\mu({\omega}) \leq M \langle  x,  x\rangle_{\mathcal{A}}.$$
    This proof that $\{T_{w}\}_{w \in \Omega}$ is a continuous K-operator frame for $End_{\mathcal{A}}^{\ast}(\mathcal{H})$.
\end{proof}
\section{Perturbation and Stability of continuous K-operator frames for $ End_{\mathcal{A}}^{\ast}(\mathcal{H})$ }
   
\begin{theorem}
	Let $\{T_{w}\}_{w \in \Omega}$ be a continuous K-operator frame for $End_{\mathcal{A}}^{\ast}(\mathcal{H})$ with bounds A and B, let $\{R_{w}\}_{w \in \Omega} \subset End_{\mathcal{A}}^{\ast}(\mathcal{H})$ and $\{\alpha_{w}\}_{w \in \Omega},\{\beta_{w}\}_{w \in \Omega} \in \mathbb{R}$ be two positively  family.
	If there exist two constants $0\leq \lambda, \mu<1$ such that for any $x \in\mathcal{H} $ we have 
\begin{align*}
    \|\int_{\Omega}&\langle (\alpha_{w}T_{w}-\beta_{w}R_{w}) x,(\alpha_{w}T_{w}-\beta_{w}R_{w}) x\rangle_{\mathcal{A}} d\mu({\omega})\|^{\frac{1}{2}}\leq \\
        & \lambda \|\int_{\Omega}\langle \alpha_{w}T_{w} x,\alpha_{w}T_{w}x\rangle_{\mathcal{A}} d\mu({\omega})\|^{\frac{1}{2}}
    +\mu\|\int_{\Omega}\langle \beta_{w}R_{w} x,\beta_{w}R_{w} x\rangle_{\mathcal{A}} d\mu({\omega})\|^{\frac{1}{2}}
\end{align*}
  	Then  $\{R_{w}\}_{w \in \Omega}$ is a continuous K-operator frame for $End_{\mathcal{A}}^{\ast}(\mathcal{H})$.
	
\end{theorem}
\begin{proof}
    For every $x \in \mathcal{H}$, we have
\begin{align*}
    \|\{\beta_{w}R_{w}x\}_{w \in \Omega}\|&\leq \|\{(\alpha_{w}T_{w}-\beta_{w}R_{w})x\}_{w \in \Omega}\| +\|\{\alpha_{w}T_{w}x\}_{w\in \Omega}\|\\
    &\leq \mu \|\{\beta_{w}R_{w}x\}_{w\in \Omega}\| + \lambda \|\{\alpha_{w}T_{w}x\}_{w\in \Omega}\|+ \|\{\alpha_{w}T_{w}x\}_{w\in \Omega}\|\\
    &=(1+\lambda) \|\{\alpha_{w}T_{w}x\}_{w\in \Omega}\|+ \mu \|\{\beta_{w}R_{w}x\}_{w \in \Omega}\|.
\end{align*}
    Then,
    $$(1-\mu)\|\{\beta_{w}R_{w}x\}_{w\in \Omega}\|\leq(1+\lambda) \|\alpha_{w}T_{w}x\|. $$
    Therefore
    $$(1-\mu)\inf_{\omega \in \Omega} (\beta_{w})\|\{R_{w}x\}_{w \in \Omega}\|\leq(1+\lambda) \sup_{\omega \in \Omega}(\alpha_{w})\|\{T_{w}x\}_{w\in \Omega}\|. $$
    Hence 
    $$\|\{R_{w}x\}_{w\in \Omega}\| \leq \frac{(1+\lambda) \sup_{\omega \in \Omega}(\alpha_{w})}{(1-\mu)\inf_{\omega \in \Omega} (\beta_{w})}\|\{T_{w}x\}_{w\in \Omega}\|.$$
    Also, for all $x \in \mathcal{H}$, we have 
\begin{align*}
    \|\{(\alpha_{w}T_{w}x\}_{w\in \Omega}\| &\leq \|\{(\alpha_{w}T_{w}-\beta_{w}R_{w})x\}_{w \in \Omega}\| +\|\{\beta_{w}R_{w}x\}_{w \in \Omega}\|\\
    &\leq \mu \|\{\beta_{w}R_{w}x\}_{w \in \Omega}\| + \lambda \|\{\alpha_{w}T_{w}x\}_{w\in \Omega}\|+ \|\{\alpha_{w}T_{w}x\}_{w\in \Omega}\|\\
    &= \lambda \|\{\alpha_{w}T_{w}x\}_{w\in \Omega}\|+(1+\mu)\|\{\beta_{w}R_{w}x\}_{w\in \Omega}\|.
\end{align*}
    then
    $$(1- \lambda)\|\{\alpha_{w}T_{w}x\}_{w\in \Omega}\| \leq (1+\mu)\|\{\beta_{w}R_{w}x\}_{w\in \Omega}\|.$$
    Hence 
    $$(1- \lambda)\inf_{\omega \in \Omega} (\alpha_{w})\|\{T_{w}x\}_{w\in \Omega}\| \leq (1+\mu)\sup_{\omega \in \Omega}(\beta_{w})\|\{R_{w}x\}_{w\in \Omega}\|.$$
    Thus
    $$\frac{(1- \lambda)\inf_{\omega \in \Omega} (\alpha_{w})}{(1+\mu)\sup_{\omega \in \Omega}(\beta_{w})}\|\{T_{w}x\}_{w\in \Omega}\| \leq \|\{R_{w}x\}_{w\in \Omega}\|.$$
    Therefore
    $$A(\frac{(1- \lambda)\inf_{\omega \in \Omega} (\alpha_{w})}{(1+\mu)\sup_{\omega \in \Omega}(\beta_{w})})^2 \|\langle x,x\rangle_{\mathcal{A}}\|\leq (\frac{(1- \lambda)\inf_{\omega \in \Omega} (\alpha_{w})}{(1+\mu)\sup_{\omega \in \Omega}(\beta_{w})})^2 \|\{T_{w}x\}_{w}\|^2\leq \|\{R_{w}x\}_{w}\|^2 .$$
  
So,  
\begin{align*}
\|\{R_{w}x\}_{w\in \Omega}\|^2&\leq (\frac{(1+ \lambda)\sup_{\omega \in \Omega}(\alpha_{w})}{(1-\mu)\inf_{\omega \in \Omega} (\beta_{w})})^2  \|\{T_{w}x\}_{w\in \Omega}\|^2 \\\leq
& B (\frac{(1+ \lambda)sup(\alpha_{w})}{(1-\mu)\inf_{\omega \in \Omega} (\beta_{w})})^2  \|\langle x,x \rangle_{\mathcal{A}}\|.
   \end{align*}
    Hence 
 \begin{align*}
    A(\frac{(1- \lambda) \inf_{\omega \in \Omega} (\alpha_{w})}{(1+\mu) \sup_{\omega \in \Omega}(\beta_{w})})^2 \|\langle x,x\rangle_{\mathcal{A}}\|&\leq \|\int_{\Omega}\langle R_{w} x,R_{w}) x\rangle_{\mathcal{A}} d\mu({\omega})\|\\
    &\leq  B (\frac{(1+ \lambda) \sup_{\omega \in \Omega}(\alpha_{w})}{(1-\mu) \inf_{\omega \in \Omega} (\beta_{w})})^2  \|\langle x,x \rangle_{\mathcal{A}}\| 
 \end{align*}

    This give that $\{R_{w}\}_{w \in \Omega}$ is a continuous K-operator frame for $End_{\mathcal{A}}^{\ast}(\mathcal{H})$.

\end{proof}
\begin{theorem}
	Let $\{T_{w}\}_{w \in \Omega}$ be a continuous K-operator frame for $End_{\mathcal{A}}^{\ast}(\mathcal{H})$ with bounds A and B. Let $\{R_{w}\}_{w \in \Omega} \in End_{\mathcal{A}}^{\ast}(\mathcal{H})$ and  $\alpha, \, \beta \geq 0$. If $0\leq \alpha + \frac{\beta}{A}< 1$ such that for all $x\in \mathcal{H}$, we have  
	$$\|\int_{\Omega}\langle (T_{w}-R_{w}) x,(T_{w}-R_{w})  x\rangle_{\mathcal{A}} d\mu({\omega})\|\leq \alpha \|\int_{\Omega}\langle T_{w} x,T_{w}  x\rangle_{\mathcal{A}} d\mu({\omega})\| +\beta \|\langle K^{\ast} x, K^{\ast}x\rangle_{\mathcal{A}}\|$$
	Then $\{R_{w}\}_{w \in \Omega}$ is a continuous K-operator frame with bounds $A\left(1-\sqrt{\alpha +\frac{\beta}{A}}\right)^2$ and $B\left(1+\sqrt{\alpha +\frac{\beta}{A}}\right)^2$.
	
\end{theorem}
\begin{proof}
	Let $\{T_{w}\}_{w \in \Omega}$ be a continuous K-operator frame with bounds A and B. Then for any $x \in \mathcal{H}$, we have
\begin{align*}
   \|\{T_{w}x\}_{w\in \Omega}\|&\leq \|\{(T_{w}-R_{w})x\}_{w\in \Omega}\| +\|\{R_{w}x\}_{w\in \Omega}\|\\
    &\leq (\alpha \|\int_{\Omega}\langle T_{w} x,T_{w}  x\rangle_{\mathcal{A}} d\mu({\omega})\| +\beta \|\langle K^{\ast} x, K^{\ast}x\rangle_{\mathcal{A}}\|)^{\frac{1}{2}}\\
    & \quad +\|\int_{\Omega}\langle R_{w} x,R_{w} x\rangle_{\mathcal{A}} d\mu({\omega})\|^{\frac{1}{2}}\\
    &\leq  (\alpha \|\int_{\Omega}\langle T_{w} x,T_{w}  x\rangle_{\mathcal{A}} d\mu({\omega})\| +\frac{\beta}{A} \|\int_{\Omega}\langle T_{w} x,T_{w}  x\rangle_{\mathcal{A}} d\mu({\omega})\|)^{\frac{1}{2}}\\
    &\quad +\|\int_{\Omega}\langle R_{w} x,R_{w} x\rangle_{\mathcal{A}} d\mu({\omega})\|^{\frac{1}{2}}\\
    &=\sqrt{\alpha +\frac{\beta}{A}} \|\{T_{w}x\}_{w\in \Omega}\|+\|\int_{\Omega}\langle R_{w} x,R_{w} x\rangle_{\mathcal{A}} d\mu({\omega})\|^{\frac{1}{2}}.
\end{align*}
    Therefore
    $$\left(1-\sqrt{\alpha +\frac{\beta}{A}}\right)\|\{T_{w}x\}_{w\in \Omega}\| \leq \|\int_{\Omega}\langle R_{w} x,R_{w} x\rangle_{\mathcal{A}} d\mu({\omega})\|^{\frac{1}{2}} .$$
    Thus 
\begin{align*}
    A\left(1-\sqrt{\alpha +\frac{\beta}{A}}\right)^2\|\langle K^{\ast} x, K^{\ast}x\rangle_{\mathcal{A}}\|
    &\leq \left(1-\sqrt{\alpha +\frac{\beta}{A}}\right)^2 \|\int_{\Omega}\langle T_{w} x,T_{w}  x\rangle_{\mathcal{A}} d\mu({\omega})\|  \\
    &\leq  \|\int_{\Omega}\langle R_{w} x,R_{w} x\rangle_{\mathcal{A}} d\mu({\omega})\|.\\
\end{align*}
\begin{align*}
    A\left(1-\sqrt{\alpha +\frac{\beta}{A}}\right)^2\|\langle K^{\ast} x, K^{\ast}x\rangle_{\mathcal{A}}\|&\leq \left(1-\sqrt{\alpha +\frac{\beta}{A}}\right)^2 \|\int_{\Omega}\langle T_{w} x,T_{w}  x\rangle_{\mathcal{A}} d\mu({\omega})\|\\
    &\leq \|\int_{\Omega}\langle R_{w} x,R_{w} x\rangle_{\mathcal{A}} d\mu({\omega})\|
\end{align*}
   
    Also, we have 
\begin{align*}
    \|\{R_{w}x\}_{w\in \Omega}\|&\leq \|\{(T_{w}-R_{w})x\}_{w\in \Omega}\| +\|\{T_{w}x\}_{w\in \Omega}\|\\
    &\leq \sqrt{\alpha +\frac{\beta}{A}} \|\{T_{w}x\}_{w\in \Omega}\|+\|\{T_{w}x\}_{w\in \Omega}\|\\
    &= \left(1+\sqrt{\alpha +\frac{\beta}{A}}\right) \|\{T_{w}x\}_{w\in \Omega}\|\\
    &\leq \sqrt{B}  \left(1+\sqrt{\alpha +\frac{\beta}{A}}\right) \|\langle x,x\rangle_{\mathcal{A}}\|^{\frac{1}{2}}.
\end{align*}
    Hence   
    $$\|\int_{\Omega}\langle R_{w} x,R_{w} x\rangle_{\mathcal{A}} d\mu({\omega})\|\leq B \left(1+\sqrt{\alpha +\frac{\beta}{A}}\right)^2 \|\langle x,x\rangle_{\mathcal{A}}\|.$$
    Therefore 
\begin{align*}
     \small A\left(1-\sqrt{\alpha +\frac{\beta}{A}}\right)^2\|\langle K^{\ast} x, K^{\ast}x\rangle_{\mathcal{A}}\|&\leq \|\int_{\Omega}\langle R_{w} x,R_{w} x\rangle_{\mathcal{A}} d\mu({\omega})\|\\
     &\leq  B \left(1+\sqrt{\alpha +\frac{\beta}{A}}\right)^2 \|\langle x,x\rangle_{\mathcal{A}}\|
\end{align*}
   Hence $\{R_{w}\}_{w \in \Omega}$ is a continuous K-operator frame with bounds $A\left(1+\sqrt{\alpha -\frac{\beta}{A}}\right)^2$ and $B\left(1+\sqrt{\alpha +\frac{\beta}{A}}\right)^2$.	
\end{proof}
\begin{corollary}
	Let $\{T_{w}\}_{w \in \Omega}$ be a continuous K-operator frame for $End_{\mathcal{A}}^{\ast}(\mathcal{H})$  with bounds A and B. Let $\{R_{w}\}_{w \in \Omega} \subset End_{\mathcal{A}}^{\ast}(\mathcal{H}) $ and $0\leq \alpha $.
	If $0\leq \alpha <A$ such that 
	$$\|\int_{\Omega}\langle (T_{w}- R_{w}) x,(T_{w}-R_{w}) x \rangle_{\mathcal{A}} d\mu({\omega})\|\leq \alpha \|\langle K^{\ast} x, K^{\ast}x\rangle_{\mathcal{A}}\|,\,\ x \in \mathcal{H},$$
	then $\{R_{w}\}_{w \in \Omega}$ is a continuous K-operator frame with bounds $A(1-\sqrt{{\frac{\alpha}{A}}})^2$ and $B(1+\sqrt{{\frac{\alpha}{A}}})^2$.
	
\end{corollary}
\begin{proof}
	The proof comes from the previous theorem.
	
\end{proof}
\begin{theorem}
	Let $\{T_{w}\}_{w \in \Omega}$ be a continuous K-operator frame for $End_{\mathcal{A}}^{\ast}(\mathcal{H})$ with bounds A and B. Let  $\{R_{w}\}_{w \in \Omega} \subset End_{\mathcal{A}}^{\ast}(\mathcal{H}) $. If ther exists $M>0$ such that for any $x \in \mathcal{H}$, we have 
\begin{equation} \label{Haja18}
 \scriptsize
 \|\int_{\Omega}\langle (T_{w}- R_{w}) x,(T_{w}-R_{w}) x \rangle_{\mathcal{A}} d\mu({\omega})\|\leq
     M min (\|\int_{\Omega}\langle T_{w} x,T_{w} x \rangle_{\mathcal{A}} d\mu({\omega})\|,\|\int_{\Omega}\langle R_{w} x,R_{w} x \rangle_{\mathcal{A}} d\mu({\omega})\|
\end{equation}
    Then $\{R_{w}\}_{w \in \Omega}$ is a continuous K-operator frame for $End_{\mathcal{A}}^{\ast}(\mathcal{H})$. The converse is true for any surjective operator K such that  \,\, $\|x\| \leq \|K^{\ast} x\| ,\,\, x \in {\mathcal{H}}$ in particular if K is co-isometry.
	
\end{theorem}
\begin{proof}
	Assume that (\ref{Haja18}) holds. \\
    On one hand, we have for any x in $\mathcal{H}$
\begin{align*}
    \sqrt{A}\|K^{\ast} x\|&\leq \|\int_{\Omega}\langle T_{w} x,T_{w} x \rangle_{\mathcal{A}} d\mu({\omega})\|^{\frac{1}{2}}\\
    &=\|\{T_{w}x\}_{w \in \Omega}\| \\
    &\leq\| \{(T_{w}-R_{w})x\}\|+\|\{R_{w}x\}\|\\
    &=\|\int_{\Omega}\langle (T_{w}-R_{w}) x,(T_{w}-R_{w}) x \rangle_{\mathcal{A}} d\mu({\omega})\|^{\frac{1}{2}}\\
    &\quad + \|\int_{\Omega}\langle R_{w} x,R_{w} x \rangle_{\mathcal{A}} d\mu({\omega})\|^{\frac{1}{2}}\\
    &\leq \sqrt{M}\|\int_{\Omega}\langle R_{w} x,R_{w} x \rangle_{\mathcal{A}} d\mu({\omega})\|^{\frac{1}{2}}+\|\int_{\Omega}\langle R_{w} x,R_{w} x \rangle_{\mathcal{A}} d\mu({\omega})\|^{\frac{1}{2}}\\
    &=(1+\sqrt{M}) \|\int_{\Omega}\langle R_{w} x,R_{w} x \rangle_{\mathcal{A}} d\mu({\omega})\|^{\frac{1}{2}}.
\end{align*}
\begin{equation}\label{Haja17}
    \sqrt{A}\|K^{\ast} x\|\leq (1+\sqrt{M}) \|\int_{\Omega}\langle R_{w} x,R_{w} x \rangle_{\mathcal{A}} d\mu({\omega})\|^{\frac{1}{2}}.
\end{equation}
    On the other hand we have
\begin{align*}
     \|\int_{\Omega}\langle R_{w} x,R_{w} x \rangle_{\mathcal{A}} d\mu({\omega})\|^{\frac{1}{2}}&=\|\{R_{w}x\}_{w \in \Omega}\| \\
     &\leq \|\{T_{w}x\}_{w \in \Omega}\| + \|\{(T_{w}-R_{w}) x\}_{w \in \Omega}\| \\
     &=\|\int_{\Omega}\langle T_{w} x,T_{w} x \rangle_{\mathcal{A}} d\mu({\omega})\|^{\frac{1}{2}}\\
     &\qquad +\|\int_{\Omega}\langle (T_{w}-R_{w}) x,(T_{w}-R_{w}) x \rangle_{\mathcal{A}} d\mu({\omega})\|^{\frac{1}{2}}\\
     &\leq \|\int_{\Omega}\langle T_{w} x,T_{w} x \rangle_{\mathcal{A}} d\mu({\omega})\|^{\frac{1}{2}}\\
     &\qquad + \sqrt{M} \|\int_{\Omega}\langle T_{w} x,T_{w} x \rangle_{\mathcal{A}} d\mu({\omega})\|^{\frac{1}{2}}\\
     &=(1+\sqrt{M}) |\int_{\Omega}\langle T_{w} x,T_{w} x \rangle_{\mathcal{A}} d\mu({\omega})\|^{\frac{1}{2}}\\
     &\leq  \sqrt{B} (1+\sqrt{M}) \|\langle x,x \rangle_{\mathcal{A}}\|
\end{align*}
     Then
\begin{equation}\label{Haja19}
     \|\int_{\Omega}\langle R_{w} x,R_{w} x \rangle_{\mathcal{A}} d\mu({\omega})\|^{\frac{1}{2}}\leq \sqrt{B} (1+\sqrt{M}) \|\langle x,x \rangle_{\mathcal{A}}\|.
\end{equation}
     From (\ref{Haja17}) and (\ref{Haja19}) we we obtain
     $$\frac{A}{ (1+\sqrt{M})^2} \|K^{\ast} x\|^2\leq \|\int_{\Omega}\langle R_{w} x,R_{w} x \rangle_{\mathcal{A}} d\mu({\omega})\|\leq B (1+\sqrt{M})^2\|\langle x,x \rangle_{\mathcal{A}}\|^2.   $$
     Hence $\{R_{w}\}_{w \in \Omega}$ is a continuous K-operator frame for $End_{\mathcal{A}}^{\ast}(\mathcal{H})$.\\
     For the converse, assume $\{R_{w}\}_{w \in \Omega}$ be a continuous K-operator frame for $End_{\mathcal{A}}^{\ast}(\mathcal{H})$ with bound C and D, and K verifie that ,\,\,i.e,\ ,\,$\|x\| \leq \|K^{\ast} x\| $. Then
     for every $x \in \mathcal{H} $, we have 
\begin{align*}
     \|\{(T_{w}-R_{w}) x\}_{w \in \Omega}\|&=\|\int_{\Omega}\langle (T_{w}-R_{w}) x,(T_{w}-R_{w}) x \rangle_{\mathcal{A}} d\mu({\omega})\|^{\frac{1}{2}}\\
     &\leq \|\{T_{w}x\}_{w \in \Omega}\|+ \|\{R_{w}x\}_{w \in \Omega}\|\\
     &= \|\int_{\Omega}\langle T_{w} x,T_{w} x \rangle_{\mathcal{A}} d\mu({\omega})\|^{\frac{1}{2}}+ \|\int_{\Omega}\langle R_{w} x,R_{w} x \rangle_{\mathcal{A}} d\mu({\omega})\|^{\frac{1}{2}}\\
     &\leq \|\int_{\Omega}\langle T_{w} x,T_{w} x \rangle_{\mathcal{A}} d\mu({\omega})\|^{\frac{1}{2}}+ \sqrt{D}\| x\|\\
     &\leq \|\int_{\Omega}\langle T_{w} x,T_{w} x \rangle_{\mathcal{A}} d\mu({\omega})\|^{\frac{1}{2}}+\sqrt{D}\|K^{\ast} x\|\\
     &\leq \|\int_{\Omega}\langle T_{w} x,T_{w} x \rangle_{\mathcal{A}} d\mu({\omega})\|^{\frac{1}{2}}+ \sqrt{\frac{D}{A}}\|\int_{\Omega}\langle T_{w} x,T_{w} x \rangle_{\mathcal{A}} d\mu({\omega})\|^{\frac{1}{2}}\\
     &=\left(1+\sqrt{\frac{D}{A}}\right)\|\int_{\Omega}\langle T_{w} x,T_{w} x \rangle_{\mathcal{A}} d\mu({\omega})\|^{\frac{1}{2}}.
\end{align*}
   Similary we can obtain
   $$\|\int_{\Omega}\langle (T_{w}-R_{w}) x,(T_{w}-R_{w}) x \rangle_{\mathcal{A}} d\mu({\omega})\|^{\frac{1}{2}}\leq \left(1+\sqrt{\frac{B}{C}}\right)\|\int_{\Omega}\langle R_{w} x,R_{w} x \rangle_{\mathcal{A}} d\mu({\omega})\|^{\frac{1}{2}}.$$
   We take
   $M=min \{1+\sqrt{\frac{D}{A}},1+\sqrt{\frac{B}{C}} \}$, then (\ref{Haja18}) is verified.
\end{proof}
\begin{theorem}
	Let $K \in End_{\mathcal{A}}^{\ast}(\mathcal{H}) $. For $ k=1,2,...,n $, let 
	$\{T_{k,w}\}_{w \in \Omega} \subset End_{\mathcal{A}}^{\ast}(\mathcal{H}) $ be a continuous K-operator frame for $End_{\mathcal{A}}^{\ast}(\mathcal{H})$ with bounds $A_k$ and $B_k$, \{$\alpha_{k}$\} be any scalars. If there exists a constant $\lambda >0$ and $p \in \{1,2,..,n\}$ such that 
	$$\lambda \|\{T_{p,w}x\}_w\|\leq \|\sum_{k=1}^{n} \{\alpha_{k} T_{k,w}x\}_w \|, x \in \mathcal{H},$$
	then $ \{\sum_{k=1}^{n} \alpha_{k} T_{k,w}x\}_{w\in \Omega}$ is  a continuous K-operator frame for $End_{\mathcal{A}}^{\ast}(\mathcal{H})$. The converse is true for every co-isometrie operator K.
\end{theorem}
\begin{proof}
	For all $x \in \mathcal{H}$, we have 
\begin{align*}
    \sqrt{A_p} \lambda \|\langle K^{\ast}x, K^{\ast}x\rangle_{\mathcal{A}}\|^{\frac{1}{2}}&\leq \lambda \|\{ T_{p,w}x\}_{w\in \Omega}\|\\
    &\leq \|\sum_{k=1}^{n} \{\alpha_{k} T_{k,w}x\}_{w\in \Omega} \|\\
    &\leq \sum_{k=1}^{n}|\alpha_{k}|\| \{ T_{k,w}x\}_{w\in \Omega} \|\\
    &\leq \max_{1 \leq k \leq n} |\alpha_{k}| \sum_{k=1}^{n}\| \{ T_{k,w}x\}_{w\in \Omega} \|\\
    &\leq \max_{1 \leq k \leq n}|\alpha_{k}| (\sum_{k=1}^{n} \sqrt{B_k}) \|\langle x, \rangle_{\mathcal{A}}\|^{\frac{1}{2}}.
\end{align*}
    Hence for any $ x \in \mathcal{H}$, we have 
    $$\sqrt{A_p} \lambda \|\langle K^{\ast}x, K^{\ast}x\rangle_{\mathcal{A}}\|^{\frac{1}{2}}\leq\|\sum_{k=1}^{n} \{\alpha_{k} T_{k,w}x\}_{w\in \Omega} \|\leq \max_{1 \leq k \leq n} |\alpha_{k}| (\sum_{k=1}^{n} \sqrt{B_k}) \|\langle x,x \rangle_{\mathcal{A}}\|^{\frac{1}{2}}. $$
    Then 
    $$A_p {\lambda}^{2} \|\langle K^{\ast}x, K^{\ast}x\rangle_{\mathcal{A}}\|\leq \|\sum_{k=1}^{n} \{\alpha_{k} T_{k,w}x\}_{w\in \Omega} \|^2\leq( \max_{1 \leq k \leq n} |\alpha_{k}|)^2 (\sum_{k=1}^{n} \sqrt{B_k})^2 \|\langle x, x\rangle_{\mathcal{A}}\|.$$
    This gives that $ \{\sum_{k=1}^{n} \alpha_{k} T_{k,w}x\}_{w\in \Omega}$ is  a continuous K-operator frame for $End_{\mathcal{A}}^{\ast}(\mathcal{H})$.
    For the converse, let K be a co-isometrie operator on $\mathcal{H}$, let  $\{\sum_{k=1}^{n} \alpha_{k} T_{k,w}x\}_w $ is a continuous K-operator frame for $End_{\mathcal{A}}^{\ast}(\mathcal{H})$ withs bounds A and B and let for all $p \in \{1,2,..,n\}$, $\{T_{k,w}x\}_w$ be a  continuous K-operator frame for $End_{\mathcal{A}}^{\ast}(\mathcal{H})$ with bounds $A_p$ and $B_p$. Then, for every $x \in\mathcal{H}$, we have 
    $$A_p\|\langle K^{\ast}x, K^{\ast}x\rangle_{\mathcal{A}}\|\leq \|\{T_{p,w}x\}_{w\in \Omega}\|^2\leq B_p \|\langle x, x\rangle_{\mathcal{A}}\|. $$
    Then 
    $$\frac{1}{B_p}\|\{T_{k,w}x\}_{w\in \Omega}\|^2\leq \|\langle x, x\rangle_{\mathcal{A}}\|.$$
    Also, we have 
    $$A\|\langle K^{\ast}x, K^{\ast}x\rangle_{\mathcal{A}}\|\leq \|\{\sum_{k=1}^{n} \alpha_{k} T_{k,w}x\}_{w\in \Omega}\|^2\leq B \|\langle x, x\rangle_{\mathcal{A}}\|.$$
    Since K is a co-isometrie operator, then
    $$\|\langle x, x\rangle_{\mathcal{A}}\|=\|\langle K^{\ast}x, K^{\ast}x\rangle_{\mathcal{A}}\| \leq \frac{1}{A}  \|\{\sum_{k=1}^{n} \alpha_{k} T_{k,w}x\}_{w\in \Omega}\|^2,\,\,\ x \in \mathcal{H} .$$
    So
    $$\frac{A}{B_p  }\|\{T_{p,w}x\}_w\|^2 \leq \|\{\sum_{k=1}^{n} \alpha_{k} T_{k,w}x\}_w\|^2,\,\,\,  x \in \mathcal{H}.$$
    Therefore, for $\lambda = \frac{A}{B_p  } $, we have 
    $$\lambda \|\{T_{p,w}x\}_{w\in \Omega}\|\leq \|\sum_{k=1}^{n} \{\alpha_{k} T_{k,w}x\}_{w\in \Omega} \|, x \in \mathcal{H}.$$
	
\end{proof}
\begin{theorem}
	Let $K \in End_{\mathcal{A}}^{\ast}(\mathcal{H}) $. For $k=1,2,...,n$, let $\{T_{k,w}\}_{w\in \Omega} \subset End_{\mathcal{A}}^{\ast}(\mathcal{H}) $ be  a continuous K-operator frame for $End_{\mathcal{A}}^{\ast}(\mathcal{H})$ withs bounds $A_k$ and $B_k$, and let  $\{R_{k,w}\} \subset End_{\mathcal{A}}^{\ast}(\mathcal{H}) $ be any family. Let S : $l^2(\mathcal{H})$ $\longrightarrow$ $l^2(\mathcal{H})$ be a bounded linear operator such that $S(\sum_{k=1}^{n}  R_{k,w}x\}_{w\in \Omega})= \{T_{p,w}x\}_{w\in \Omega}$, for some $p \in \{1,2,...,n\}$. If there exists a constant $\lambda >0$ such that 
	
	$$ \|\int_{\Omega}\langle (T_{k,w}- R_{k,w}) x,(T_{k,w}-R_{k,w}) x \rangle_{\mathcal{A}} d\mu({\omega})\|$$
	 $$\leq \lambda  \|\int_{\Omega}\langle T_{k,w} x,T_{k,w} x \rangle_{\mathcal{A}} d\mu({\omega})\|,x \in \mathcal{H}, k=1,2,...,n. $$
	Then $\{\sum_{k=1}^{n}  R_{k,w}\}_{w\in \Omega}$ is a continuous K-operator frame for $End_{\mathcal{A}}^{\ast}(\mathcal{H})$ .
\end{theorem} 
\begin{proof}
	For every $x \in \mathcal{H}$, we have 
\begin{align*}
    \|\{\sum_{k=1}^{n}  R_{k,w}x\}_{w\in \Omega}\|&\leq \sum_{k=1}^{n} \|\{ R_{k,w}x\}_{w\in \Omega}\|\\
    &\leq \sum_{k=1}^{n} (\|\{ (T_{k,w}-R_{k,w})x\}_{w\in \Omega}\|+\| \{T_{k,w}x\}_{w\in \Omega}\|)\\
    &\leq (\sqrt{\lambda}\| \{T_{k,w}\}x\|+ \| \{T_{k,w}x\}_{w\in \Omega}\|)\\
    &=(1+\sqrt{\lambda}) \sum_{k=1}^{n}\| \{T_{k,w}\}x\|\\
    &\leq (1+\sqrt{\lambda}) (\sum_{k=1}^{n}\sqrt{B_k}) \|\langle x,x \rangle_{\mathcal{A}}\|^{\frac{1}{2}}.
\end{align*}
    We have for any $x \in \mathcal{H}$
    $$\|S(\sum_{k=1}^{n}  R_{k,w}x\}_{w\in \Omega})\|=\| \{T_{p,w}x\}_{w\in \Omega}\|.$$
    Hence
\begin{align*}
    \sqrt{A_p}\|\langle K^{\ast}x, K^{\ast}x\rangle_{\mathcal{A}}\|^{\frac{1}{2}}&\leq \|\{T_{p,w}x\}_{w\in \Omega}\|\\
    &=\|S(\sum_{k=1}^{n}  R_{k,w}x\}_{w\in \Omega})\|\\
    &\leq \|S\|.\|\{\sum_{k=1}^{n}  R_{k,w}x\}_{w\in \Omega}\|, x \in \mathcal{H}
\end{align*} 
    So
    $$\frac{\sqrt{A_p}}{\|S\|}\|\langle K^{\ast}x, K^{\ast}x\rangle_{\mathcal{A}}\|^{\frac{1}{2}}\leq \|\{\sum_{k=1}^{n}  R_{k,w}x\}_{w\in \Omega}\|, x \in \mathcal{H}.$$
    Therefore 
    $$\frac{\sqrt{A_p}}{\|S\|}\|\langle K^{\ast}x, K^{\ast}x\rangle_{\mathcal{A}}\|^{\frac{1}{2}}\leq \|\{\sum_{k=1}^{n}  R_{k,w}x\}_{w\in \Omega}\|\leq (1+\sqrt{\lambda}) (\sum_{k=1}^{n}\sqrt{B_k}) \|\langle x,x \rangle_{\mathcal{A}}\|^{\frac{1}{2}},x \in \mathcal{H}. $$
    Thus 
    $$\frac{A_p}{\|S\|^2}\|\langle K^{\ast}x, K^{\ast}x\rangle_{\mathcal{A}}\|\leq \|\{\sum_{k=1}^{n}  R_{k,w}x\}_{w\in \Omega}\|^2 \leq (1+\sqrt{\lambda})^2 (\sum_{k=1}^{n}\sqrt{B_k})^2 \|\langle x,x \rangle_{\mathcal{A}}\|,x \in \mathcal{H}.  $$
    This gives that $\{\sum_{k=1}^{n}  R_{k,w}\}_{w\in \Omega}$ is a continuous K-operator frame for $End_{\mathcal{A}}^{\ast}(\mathcal{H})$ .

\end{proof}
{\bf Conflict of interest:}
On behalf of all authors, the corresponding author states that there is no conflict of interest.

{\bf Data availability statement:}  No data were used to support this study.

{\bf Funding statement:}  This study was not funded.
\bibliographystyle{amsplain}

%\vspace{0.1in}
%\hrule width \hsize \kern 1mm
%\hrule width \hsize height 2pt
\end{document}